\documentclass[a4paper,12pt]{article}

\usepackage{amsfonts, bm, amsmath, cases}
\usepackage{cleveref}
\usepackage{mathrsfs} 
\usepackage{graphicx}
\usepackage{color}
\usepackage[normalem]{ulem}
\usepackage{bm}

\usepackage{booktabs,comment}
\usepackage[textwidth=18mm]{todonotes}

\newtheorem{theorem}{Theorem}
\newtheorem{lemma}{Lemma}
\newtheorem{proposition}{Proposition}
\newtheorem{corollary}{Corollary}
\newtheorem{algorithm}{Algorithm}
\newtheorem{definition}{Definition}
\newtheorem{remark}{Remark}
\newtheorem{problem}{Problem}
\newtheorem{example}{Example}
\newenvironment{proof}{\paragraph{Proof:}}{\hfill$\square$}

\def\supp{\text{supp}}
\def\N{\mathbb{N}}
\def\Z{\mathbb{Z}}
\def\Q{\mathbb{Q}}
\def\R{\mathbb{R}}
\def\a{{{\alpha}}}
\def\c{{{\bm{c}}}}
\def\ba{{\bm{\alpha}}}
\def\bb{{\bm{\beta}}}

\def\x{{\bm{x}}}
\def\y{{\bm{y}}}

\def\pt{{\bm{a}}}
\def\spt{{\bm{u}}}

\def\gdeg{{d_{\g}}}

\usepackage[noend]{algpseudocode}      %%% this is for the algorithms
\algnewcommand{\algorithmicassumption}{\textbf{Requirement:}}
\algnewcommand{\Assume}{\item[\algorithmicassumption]}
\algrenewcomment[1]{\hfill \(\triangleright\) \emph{\small #1}} % to italicize text in comments
\algnewcommand{\InlineIf}[2]{ % single line if-then
  \State \algorithmicif\ #1\ \algorithmicthen\ #2}
\algnewcommand{\InlineIfElse}[3]{ % single line if-then-else
  \State \algorithmicif\ #1\ \algorithmicthen\ #2\ \algorithmicelse\ #3}
\algnewcommand{\InlineFor}[2]{\algorithmicfor\ #1\ \algorithmicdo\ #2} % single line for loop
\newcommand\g{\textbf{\textit{g}}}

\newcommand{\nneg}[1]{{\mathscr{P}({#1})}}

\newcommand{\fdef}[1]{{\it #1}}
\newcommand{\riesz}{{\mathscr{L}}}
\newcommand{\algoname}[1]{{\normalfont\textbf{#1}}}
\newcommand{\ud}{\mathrm{d}}

\title{\bf Algebraic certificates for the truncated moment problem}

\begin{document}

\author{
	Didier Henrion$^{1,2}$,
	Simone Naldi$^3$,
	Mohab Safey El Din$^4$
}
\footnotetext[1]{CNRS, LAAS, Universit\'e de Toulouse, France. }
\footnotetext[2]{Faculty of Electrical Engineering, Czech Technical University in Prague, Czechia.}
\footnotetext[3]{Universit\'e de Limoges, CNRS, XLIM, Limoges, France.}
\footnotetext[4]{Sorbonne Université, LIP6, Paris, France.}

\date{February 13, 2023}

\maketitle

\begin{abstract}
  The truncated moment problem consists of determining whether a given finite-dimensional vector of real numbers $\y$ is obtained by integrating a basis of the vector space of polynomials of bounded degree with respect to a non-negative measure on a given set $K$ of a finite-dimensional Euclidean space. This problem has plenty of applications {\it e.g.} in optimization, control theory and statistics. When $K$ is a compact semialgebraic set, the duality between the cone of moments of non-negative measures on $K$ and the cone of non-negative polynomials on $K$ yields an alternative: either $\y$ is a moment vector, or $\y$ is not a moment vector, in which case there exists a polynomial strictly positive on $K$ making a linear functional depending on $\y$ vanish. Such a polynomial is an {\it algebraic certificate} of moment unrepresentability. We study the complexity of computing such a certificate using computer algebra algorithms.
\end{abstract}

\section{Introduction}

\paragraph{Problem statement}
Let $\x = (x_1,\ldots,x_n)$ be variables, $\R[\x]$ be the ring of $n$-variate
real polynomials and for $d \in \N$, let $\R[\x]_{\leq d}$ be the vector space
of real polynomials of degree at most $d$. The multivariate monomial with
exponent $\ba = (\a_1,\ldots,\a_n) \in \N^n$ is denoted by $\x^\ba := x_1^{\a_1}
\cdots x_n^{\a_n}$, and its total degree by $|\ba| = \a_1+\cdots+\a_n$. For $\g
= (g_1,\ldots,g_k) \in \R[\x]^k$, the \fdef{basic semialgebraic set} associated
with $\g$ is
\begin{equation}
  \label{semialgset}
S(\g) = \{\pt \in \R^n : g_1(\pt) \geq 0, \ldots, g_k(\pt) \geq 0\}.
\end{equation}
Given $n,d \in \N$ and a sequence of real numbers $\y = (y_\ba)_{\ba \in \N^n_d}$
indexed by $\N^n_d = \{\ba \in \N^n : \sum_{i=1}^n\alpha_i\leq d\}$, the {\it
  truncated moment problem} (below TMP) is the question of deciding whether
there exists a nonnegative Borel measure $\mu$ on $\R^n$, with support in $K =
S(\g)$, and such that
\begin{equation}
  \label{moment}
  y_\ba = \int_K x^\ba \, d\, \mu, \,\,\,\,\,\,\,\,\, \text{ for all } \ba \in \N^n_d,
\end{equation}
%{\it cf.} \Cref{TMP}.
If this is the case, one says that $\y$ is {\it moment-re\-pre\-sen\-ta\-ble on $K$}.
More generally, the monomial basis can be replaced by another linear basis of $\R[\x]_{\leq d}$,
({\it e.g.} Chebyshev polynomials). The TMP is the truncated version of the classical
{\it full moment problem} \cite{kuhlmann2002positivity,s17}.

\paragraph{Overview and state of the art}
The TMP is of central importance in data science. It is at the heart of several
questions in optimization, control theory or statistics, to mention just a few
application domains. It is a key ingredient to the moment-SOS (sum of squares)
approach \cite{hkl20} which consists of solving numerically non-convex
non-linear problems at the price of solving a family of finite-dimensional conic
optimization (typically semidefinite programming) problems. Mathematical
foundations of the moment problem were recently surveyed in \cite{s17,f16}.

The TMP can be interpreted as a {\em decision problem of the first order theory
  of the reals}, in which case, the
input-output data structure is as follows. The input is encoded by a
finite-di\-men\-sio\-nal real vector $\y$, whose coordinates are indexed by a
basis of the vector space of
$n$-variate polynomials of degree $\leq d$, together with finitely-many
polynomial inequalities defining a basic semialgebraic set $K \subset \R^n$.
The output is a decision yes/no whether $\y$ is obtained by integrating the
basis with respect to a non-negative measure supported on $K$.

When $K$ is compact, the TMP is dual, in the convex analysis sense, to the
problem of determining whether a polynomial is non-negative on the
semialgebraic set $K$. This latter problem is at the heart of the development
of real algebra in the twentieth century. Whereas deciding nonnegativity of a
polynomial of degree at least four can be challenging (NP-complete problems can
be cast as instances of such positivity problems \cite{l09}), there exist
algebraic certificates based on SOS for deciding strict positivity under
compactness-like assumptions that can be computed by solving a semidefinite
programming (SDP) problem \cite{powers1998algorithm}.
For positivity over compact basic semialgebraic sets, these certificates have
the form of linear combinations with SOS
coefficients of polynomials that are explicitly nonnegative on $K$
\cite{putinar1993positive,schmudgen1991}. The size of the SDP is determined by the degree
of the SOS-representation. Seminal papers
\cite{ramana1997exact,porkolab1997complexity} have started to investigate the
complexity of SDP in the context of rational arithmetic. More recent work is based
on the determinantal structure of semidefinite programs \cite{henrion2016exact}.

The specific case study of computation of SOS certificates have recently received
a lot of attention from the computer algebra community \cite{KLYZ08, kaltofen2012exact,henrion2016exact,henrion2019spectra,magron2021exact} especially for the
question whether certificates exist over the rational numbers
\cite{peyrl2008computing, SaZhi10, guo2013computing,scheiderer2016sums}.
In this work we make use of quantifier elimination for determining
bounds on the complexity of computing certificates for unrepresentability by
using quantitative results from~\cite{BPR96}. These
also rely on recent advances in the complexity analysis of Putinar
Positivstellensatz \cite{baldi2022effective}.

Note that these non-negativity problems can be solved with computer algebra
algorithms which are root-finding algorithms, hence which do not provide
algebraic certificates of non-negativity but do provide witnesses (real points)
of negativity whenever they exist. The first family of algorithms for doing so
is based on the so-called Cylindrical Algebraic Decomposition~\cite{Collins}. It
has complexity which is doubly exponential in the number of variables, which
would lead, in the context of TMP, to complexity bounds that are doubly exponential
in $n + \tbinom{n+d}{d}$.
The second family of algorithms, named critical point method, initiated by~\cite{GV88},
has complexity which is singly exponential in the number of variables~(see
\cite{BPR} and references therein).

The purpose of this communication is to leverage on these achievements and initiate the study and development of {\it computer algebra algorithms for solving the truncated moment problem for basic semialgebraic sets}.

%%%%%When $K$ is moreover compact, the duality between the cone of moments of non-negative measures on $K$ and the cone of non-negative polynomials on $K$ yields an alternative: either $\y$ is a moment vector represented by a non-negative measure on $K$, or $\y$ is not a moment vector, in which case there exists a polynomial strictly positive on $K$ making a linear functional of $\y$ vanish. Such a polynomial is an {\it algebraic certificate} of moment unrepresentability. We study the complexity of computing such a certificate using computer algebra algorithms.

%\paragraph{Problem statement}

%We also define , and
%$\N^n_d = \{\ba \in \N^n : |\ba| \leq d\}$, and recall that $\dim(\R[\x]_{\leq d})=
%\lvert \N^n_d \rvert = \tbinom{n+d}{d}$.

\paragraph{Overview of the contribution}
In the wake of the mentioned duality with moments, the existence of
SOS-certificates for nonnegative polynomials in the dual side, suggests
that similar certificates might be used for the TMP on the primal side.

On the
one hand, when a measure exists whose partial moments coincide with $\y$, the
measure itself is the natural algebraic proof that allows the user to verify
directly that $\y$ is moment-representable. On the other hand, this paper shows the existence
of explicit algebraic certificates of unrepresentability: these have the form of
positive polynomials on $K$ admitting a positivity certificate and orthogonal
to the vector $\y$.

Our contribution is based on the fact that the TMP, as a decision problem, is equivalent to the
feasibility of a convex conic program in a finite dimensional vector space. More precisely, the question
is whether the interior of $\nneg{K}_d$, the cone of polynomials nonnegative on $K$ of degree at most $d$,
intersects the vanishing locus of the Riesz functional $\riesz_\y : \R[x]_{\leq d} \to \R$ defined by
$\riesz_\y\left(\sum p_\ba x^\ba\right) = \sum p_\ba y_\ba.$

When $K$ is compact, Tchakaloff's Theorem \cite{t57} states that $\y$ is moment-re\-pre\-sen\-ta\-ble
whenever $\riesz_\y$ is nonnegative on $\nneg{K}_d$, in other words, if the mentioned conic
program is {\it weakly feasible}. In this case there exists an atomic measure $\mu = \sum_{i=1}^s c_i
\delta_{x_i}$ whose moment sequence of degree $\leq d$ is $\y$: such measure is a (real) solution
of a highly structured polynomial system of type multivariate Vandermonde, which we do not investigate here.
On the other side of the coin, $\y$ is not moment-representable exactly when the conic program
is {\it strongly feasible}: in algebraic terms, this means that there exists a polynomial $p \in \nneg{K}_d$,
(strictly) positive on $K$, in the kernel of $\riesz_\y$.

In our contribution we study algorithmic aspects of the computation of such {\it unrepresentability
  algebraic certificates} when $\y$ is not moment-representable. First, we show that if the quadratic module
corresponding to the description of $K$ is archimedean, such certificates exist. We define an integer
invariant called the unrepresentability degree which measures the complexity of computing such certificate.
We give bounds on such degree that only depend on the input size of our algorithm.
When the input vector $\y$ is defined over $\Q$, and if it is not moment-representable, we show that there
exists a rational certificate of unrepresentability.

%\paragraph{Outline of the paper}
%\stodo{small paragraph on the structure of the paper}

\section{Preliminaries}

%%\subsection{General notation}

%\subsection{Cone of positive polynomials}
\subsection{Nonnegative polynomials}
Let $K \subset \R^n$. A polynomial $f \in \R[\x]$ is called \fdef{nonnegative} on
$K$ if $f(\pt) \geq 0$ for all $\pt \in K$ and \fdef{positive} if
$f(\pt) > 0$ for all $\pt \in K$. We denote by
$$
\nneg{K}_d = \{f \in \R[\x]_{\leq d} : f(\pt) \geq 0, \,\forall\pt \in K\}
$$
%%%and $\nneg{K} = \cup_{d \geq 0} \nneg{K}_d$\footnote{Check if needed}.
%%%Both $\nneg{K}$ and $\nneg{K}_d$ are convex cones.
the convex cone of polynomials of degree $\leq d$, nonnegative on $K$.
If $K$ is semialgebraic, then $\nneg{K}_d$ is also semialgebraic by {the
  theorem of Tarski on} quantifier elimination {over the reals} \cite{Tarski}.
%%%%but not basic in general.

We denote by $\Sigma_n \subset \R[\x]$ the \fdef{cone of sums of squares} of polynomials, and by
$\Sigma_{n,2d} = \Sigma_n \cap \R[\x]_{\leq 2d}$ its degree-$2d$ part (remark that
$\Sigma_{n,2d+1}=\Sigma_{n,2d}$). The cone $\Sigma_{n,2d}$ is full-dimensional in
$\R[\x]_{\leq 2d}$ and is contained in $\nneg{\R^n}_{2d}$.

%%%In particular, the inequalities defining $\nneg{K}_d$, computable by Tarski's Quantifier Elimination \cite{Tarski}, cannot be easily\todo{\scriptsize I would be more precise} obtained from the inequalities of $K$. Determining tractable subsets\todo{\scriptsize How do we define ``tractable subsets''?} of $\nneg{K}$ is a crucial question of real algebraic geometry and polynomial optimization.

%\subsection{Quadratic modules}

Testing membership in cones of nonnegative polynomials over semialgebraic sets can be challenging.
Indeed, testing nonnegativity of polynomials of degree $\geq 4$ is NP-hard \cite[Sec.~1.1]{l09}.
Nevertheless $\nneg{K}_d$ contains subcones that can be represented via linear matrix inequalities,
thus testing membership in these subcones can be cast as a \fdef{semidefinite programming (SDP) problem}.

Examples of such subsets are quadratic modules: for $K = S(\g)$ as in \eqref{semialgset} and $d \in \N$, and denoting
$g_0:=1$, we define the \fdef{quadratic module} associated with $\g$
and its \fdef{truncation of order $d$} respectively
by:
\begin{align*}
  Q(\g) &= \left\{\sum_{i=0}^k \sigma_i g_i : \sigma_i \in \Sigma_n\right\} \,\,\,\,\,\,\text{ and}\\
  Q(\g)[d] &= \left\{\sum_{i=0}^k \sigma_i g_i : \sigma_i \in \Sigma_n, \,\deg(\sigma_i g_i) \leq d\right\}.
\end{align*}
If $f \in Q(\g)$, we call the polynomials $[\sigma_0,\sigma_1,\ldots,\sigma_k]$ in an expression
$f = \sum_{i=0}^k \sigma_i g_i$ a \fdef{SOS-certificate for $f \in Q(\g)$}.
The sets $Q(\g)$ and $Q(\g)[d]$ depend on the polynomials $\g$ in the description of $K$.
Denoted by $Q(g)_d := Q(\g) \cap \R[\x]_{\leq d}$, by construction one has
$Q(\g)[d] \subset Q(\g)_d$ but, in general, this inclusion is strict: in other words, for some
$\g$ there exists a polynomial $f \in Q(\g)$, of degree $d$, such that in any certificate
$f = \sum_{i=0}^k \sigma_i g_i$, at least one product $\sigma_i g_i$ has degree $>d$.

It is not even true that $Q(\g)_d \subset Q(\g)[D]$ for possibly large $D$: these quadratic modules
are called stable.

\begin{definition}[{\cite[Sec.~4.1]{marshall2008positive}}]
  For $d \in \N$, a quadratic module $Q(\g)$ is called \emph{stable in degree $d$} if there exists $D$
  such that $Q(\g)_d \subset Q(\g)[D]$. It is called \emph{stable} if it is stable
  in every $d \in \N$ (we call the function $D=D(d)$ a \emph{stability function} for $Q(\g)$).
\end{definition}
Stability in a given degree depends on the generators $\g$ whereas stability depends only on the quadratic
module $Q(\g)$ and is equivalent to the existence of degree bounds for the representation
$f = \sum_i \sigma_i g_i \in Q(\g)$ that only depend on the degree of $f$, see for instance
\cite{netzer2009stability}.
One example of stable quadratic module is the cone $\Sigma_{n,2d} = Q(0)_{2d}$: indeed, every polynomial in
$\Sigma_{n,2d}$ is a sum of squares of polynomials of degree at most $d$, that is, $Q(0)_{2d} = Q(0)[2d]$ and
hence $D(2d)=D(2d+1)=2d$ is a stability function for $\Sigma_n$.

As well as $\Sigma_{n,2d}$, truncated quadratic modules are
\fdef{semidefinite representable sets}, that is linear images of feasible sets of
SDP problems (also known as {projected spectrahedra} or spectrahedral
{shadows} in the literature): given a description $\g$ for $K = S(\g)$,
computing one polynomial in $Q(\g)[D]$ amounts to solving a single SDP problem ({\it cf.}
\Cref{ssec:naif_algo}).

\begin{definition}[{\cite{putinar1993positive}}]
  A quadratic module $Q(\g)$ is called \fdef{ar\-chi\-me\-de\-an} if there exists $u \in Q(\g)$ such that $S(u)$
  is compact.
\end{definition}

Remark that if $Q(\g)$ is archimedean, then $S(\g) \subset S(u)$ thus $S(\g)$ is compact.
Archimedeanity and stability are often mutually exclusive properties, indeed, for $n \geq 2$,
an archimedean quadratic module is not stable.

\begin{theorem}[Putinar's Positivstellensatz \cite{putinar1993positive}]
  \label{putinar}
  Let $K = S(\g)$ be non empty, and assume $Q(\g)$ is archimedean. Then every polynomial positive on $K$
  belongs to $Q(\g)$.
\end{theorem}

The problem of bounding the degree $D$ of the summands in a Putinar certificate $f=\sum_{i=0}^k \sigma_i g_i$
for a polynomial $f \in \nneg{K}_d$, is called the \fdef{effective} Putinar Positivstellensatz,
see \cite{nie2007complexity,baldi2022effective}. The work \cite{baldi2022effective} gives a bound for $D$ as a
function of $d,n$, of the polynomial $f$ and of geometrical parameters of $K$, see \cite[Th.~1.7]{baldi2022effective},
{\it cf.} \Cref{bound_unrepr_deg}.

\subsection{Moments}

A \fdef{nonnegative Borel measure} (below \fdef{measure}, for short) $\mu$
is a bounded nonnegative linear functional on the $\sigma$-algebra of Borel
sets $\mathscr{B}(\R^n)$. The \fdef{support} of $\mu$ is the complement
of the largest open Borel set $A \in \mathscr{B}(\R^n)$ such that $\mu(A) = 0$,
denoted by $\supp(\mu)$.

Let $K \subset \R^n$ be a Euclidean closed set. A measure $\mu$ is
\fdef{supported} on $K$ if $\supp(\mu) \subset K$
(in particolar it satisfies $\mu(\R^n \setminus K) = 0$).
For $\spt \in K$, we denote by $\delta_{\spt}$ the Dirac measure supported on the
singleton $\{\spt\}$. A finite linear combinations $\sum_{i=1}^s c_i \delta_{\spt_i}$
of $s$ Dirac measures is called $s$-\fdef{atomic}: its support is
$\supp(\sum_{i=1}^s c_i \delta_{\spt_i})=\{\spt_1,\ldots,\spt_s\}$.

For $\ba \in \N^n$, the \fdef{(monomial) moment of exponent $\ba$} of $\mu$ is
the real number $\int_K \x^{\ba} \, \ud\mu(\x)$ as in \eqref{moment}.
We say that $\mu$ satisfying \eqref{moment} is a \fdef{representing measure} for the
sequence $\y=(y_\ba)_{\ba \in \N^n_d}$.
If $K$ is compact, the Stone-Weierstrass Theorem implies that a measure is uniquely determined
by its (infinite-dimensional) sequence of monomial moments,
see \cite[Cor.~3.3.1]{marshall2008positive}.
%%%\todo{\scriptsize Not so obvious to me}.

In this work, we address the following inverse problem for semialgebraic sets:

\begin{problem}[Truncated Moment Problem {\cite[Ch.~17-18]{s17}}]
  \label{TMP}
  Let $K \subset \R^n$ be a basic closed semialgebraic set. Given a finite sequence
  $\y=(y_{\ba})_{\ba \in \N^n_d}$ of real numbers, determine whether $\y$
  admits a representing measure supported on $K$.
\end{problem}

For $\y = (y_{\ba})_{\ba \in \N^n}$, the matrix $M_d(\y) = (y_{\ba+\bb})_{\ba,\bb \in \N^n_d}$
is called the \fdef{moment matrix of order $d$} of $\y$.
%%%, and for $f \in \R[\x]$, the matrix $M_d(f\y) = (y_{\ba+\bb})_{\ba,\bb \in \N^n_d}$ is called the \fdef{localizing matrix of $f$ of order $d$}.
We recall that we denote by $\riesz_\y : \R[\x]_{\leq d} \to \R$
the \fdef{Riesz functional associated with $\y$}, defined by $\riesz_\y(\x^\ba)=y_{\ba}$ and
extended linearly on $\R[\x]_{\leq d}$.
%\[
%\begin{array}{rclll}
%  \riesz_\y & : & \R[\x]_{\leq d}                     & \to     & \R \\
%            &   & p=\sum_{\ba \in \N^n_d}p_\ba \x^\ba & \mapsto & \sum_{\ba \in \N^n_d}p_\ba y_\ba
%\end{array}
%\]

Let now $K = S(\g)$ be a basic closed semialgebraic set, and let
\begin{align*}
  \mathscr{M}(K)_d
  = \Big\{& \y = (y_{\ba})_{\ba \in \N^n_d} \in \R^m : \exists\,\mu, \supp(\mu) \subset K, \\
  & \forall\ba\in\N^n_d, \, y_{\ba} = \int_K \x^{\ba} d\mu(\x)\Big\} 
\end{align*}
denote the set of moments of order up to $d$ of nonnegative Borel measures with support in $K$. The
set $\mathscr{M}(K)_d$ is in general not closed, as shown by the following example.

\begin{example}[{\cite[Rem.~3.147]{BPT}}]
  \label{example_not_closed}
  Let $n=1, d=4$ and $\y=(1,0,0,0,1)$.
  Its second moment matrix $M_2(\y)$
  %%%%= \left(\begin{matrix} 1 & 0 & 0 \\ 0 & 0 & 0 \\ 0 & 0 & 1\end{matrix}\right) $$
  is positive semidefinite but
  the vector $\y$ is not representable by a nonnegative univariate measure, indeed $y_2 = 0$ but
  $y_4 \neq 0$. Thus $\y \not\in \mathscr{M}(\R)_4$. Nevertheless $\y \in \overline{\mathscr{M}(\R)_4}$,
  the Euclidean closure of $\mathscr{M}(\R)_4$: indeed $\y = \lim_{\epsilon \to 0} \y_\epsilon$
    where $\y_\epsilon = (1,0,\epsilon^2,0,1)$ is the sequence of moments of degree $\leq 4$ of the
    $3-$atomic measure
    $$
    \mu_\epsilon = \frac{\epsilon^4}{2} \left(\delta_{\frac{1}{\epsilon}}+\delta_{-\frac{1}{\epsilon}}\right)+(1-\epsilon^4)\delta_0
    $$
  \hfill$\Box$
\end{example}

Let $V^\vee$ be the {dual vector space} of a real vector space $V$, that
is the set of $\R$-linear functionals $L : V \to \R$. If $C \subset V$ is a convex cone, the set
$C^* = \{L \in V^\vee : L(\pt) \geq 0, \,\forall\pt \in C\}$ is called the {dual cone} of $C$.
It is straightforward to see that $\mathscr{M}(K)_d \subset \nneg{K}_d^*$, and a non-trivial result
is that equality holds for $K$ compact:

\begin{theorem}[{Tchakaloff's Theorem \cite{t57}}]
\label{tchakaloff}
Let $K \subset \R^n$ be compact and $d \in \N$. Then
\[
\mathscr{M}(K)_d = \nneg{K}_d^* = 
\{\y \in \R^m : \riesz_\y(p) \geq 0, \:\forall p \in \nneg{K}_d\}.
\]
\end{theorem}

\Cref{tchakaloff} is a finite-dimensional version of the Riesz-Haviland Theorem
\cite{cf08}. It is often used in the moment-SOS hierarchy, see \cite[Lemma
1.7]{hkl20}. A modern statement and proof can be found {\it e.g.} in \cite[Theorem
5.13]{l09}, see also \cite{blekhermanCoreVariety}.

\section{An algorithm for the moment problem}

We describe an algorithm based on semidefinite programming that solves Problem \ref{TMP}
for compact basic semialgebraic sets.

To do that, we first give a characterization of the interior
of cones of nonnegative polynomials on compact sets $K \subset \R^n$ (\Cref{ssec:interior}).
Next we interpret \Cref{TMP} as a conic feasibility problem and prove that its
solvability is related to the feasibility type of the program (\Cref{ssec:conic}). Finally we describe
our algorithm in \Cref{ssec:naif_algo}.

\subsection{Interior of $\nneg{K}_d$}
\label{ssec:interior}
Let $K \subset \R^n$ be non-empty.
%%%and by $\nneg{K} \subset \R[x]$ the cone of polynomials nonnegative on $K$.
For $f \in \R[\x]$, we denote by $f^* := \inf_{x \in K} f(x)$, possibly $-\infty$.

It is straightforward to construct examples of positive polynomials $f \in \nneg{K}$,
on a non-compact set $K$, such that $f \not\in \text{Int}(\nneg{K})$ even if $f^* > 0$
(for instance $1 = \lim_{\epsilon \to 0} 1-\epsilon x$ so $1 \not\in \text{Int}(\nneg{\R}_2)$).
More generally, positive sum-of-squares polynomials of degree $<d$ lie in
the boundary of the cone of nonnegative polynomials of degree $\leq d$
({\it cf.} \cite[\S 4.4.3]{BPT}).

%\begin{example}
%  \label{rr}
%  Let $K = \{(x,y) \in \R^2 : x \geq 0, xy \geq 1\}$. Then the infimum of the polynomial
%  $f = x+1$ on $K$ is equal to $1$, and $f$ is the pointwise limit of polynomials $f_\epsilon = f-\epsilon y$ for
%  $\epsilon \to 0^+$. However $f_\epsilon \not\in \nneg{K}$ for all $\epsilon>0$, hence this shows
%  $f \not\in \text{Int}(\nneg{K})$. In particular for this example
%  $\text{Int}(\nneg{K}) \neq \{f \in \R[x] : f^* > 0\}$.
%%%%  This has something to do with ``stability of infeasibility'' as defined in \cite{naldi2021conic}.
%  \hfill$\blacksquare$
%\end{example}

%\begin{example}
%  Let $K = \R$ and $f = 1+x^2 \in \R[\x]_{\leq 4}$. Then $f>0$ on $\R$ but
%  $f \not\in \text{Int}(\nneg{\R}_4)$: indeed for all $\epsilon \neq 0$, the polynomial
%  $f_\epsilon = f+\epsilon x^3 \in \R[\x]_{\leq 4}$ will take a negative value on $\R$,
%  hence $f_\epsilon \not\in \nneg{K}_4$, whereas $f_0 = f$. Remark that $f \in
%  \text{Int}(\nneg{\R}_2)$, indeed its Gram matrix as a degree-two
%  sum of squares is $\left[\begin{smallmatrix} 1 & 0 \\ 0 & 1\end{smallmatrix}\right]$.
%  \hfill$\blacksquare$
%\end{example}

The next folklore lemma shows that, for compact sets, the interior of $\nneg{K}_d$
is exactly the set of positive polynomials over $K$.

\begin{lemma}
  \label{lem:interior}
  Let $K \subset \R^n$ be non-empty, and let $d \in \N$.
  %%%  $\nneg{K} = \{f \in \R[x] : f^* \geq 0\}$
Then $\text{Int}(\nneg{K}_d) \subset \{f \in \R[\x]_{\leq d} : f^* > 0\}$. If $K$ is compact, equality holds, and $\text{Int}(\nneg{K}_d)$ consists of exactly those polynomials in $\R[\x]_{\leq d}$ that are positive on $K$.
%%    $\text{Int}(\nneg{K}_d) = \{f \in \R[x]_{\leq d} : f^* > 0\}$.
%%%  \end{enumerate}
\end{lemma}
\begin{proof}
  %%  {\it Item (1)}. This is by definition of nonnegative polynomials and by continuity.
  If $f \in \R[\x]_{\leq d}$ is such that $f^* \leq 0$, then $(f-\epsilon)^* = f^*-\epsilon < 0$ for all $\epsilon > 0$, thus
  %%%\todo{\scriptsize ??}
%%%%  \footnote{SIMONE: Je pense que c'est correct maintenant, pas besoin d'etre semialgebrique: simplement, l'infimum d'un ensemble de valeurs non-negatives ne peut pas être negatif.}
  $f-\epsilon \not\in \nneg{K}_d$ for all $\epsilon > 0$, hence $f \not\in \text{Int}(\nneg{K}_d)$, which proves the sought inclusion $\subset$.
  
  Now assume $K$ is compact, and let $f \not\in \text{Int}(\nneg{K}_d)$. Then $f$ is  in the closure of the complement of $\nneg{K}_d$ in $\R[\x]_{\leq d}$, that is, $f$ is the  pointwise limit $f = \lim_{k \to \infty} f_k$ of polynomials $f_k \not\in \nneg{K}_d$, in particular, satisfying $f_k^* < 0$ for all $k$.  Since $K$ is compact, $f_k^* = \min_{x \in K} f_k(x) = f_k(x_k)$  for some $x_k \in K$. Let $\overline{x} \in K$ be a limit point of $\{x_k\}_k$, which exists by Bolzano-Weierstrass Theorem. Thus up to extracting a subsequence, one has $0 \geq \lim_{k} \, f_k^* = \lim_{k} f_k(x_k) = f(\overline{x}) \geq f^*$, which shows the inclusion $\supset$, thus the equality $\text{Int}(\nneg{K}_d) = \{f \in \R[\x]_{\leq d} : f^* > 0\}$. Since the infimum of a polynomial function on a compact set is its minimum, one has $f^* > 0$ if and only if $\min_{x \in K} f(x) > 0$ if and only if $f$ is positive on $K$, as claimed.
\end{proof}

\begin{remark}
  \Cref{putinar} and \Cref{lem:interior} ensure that if $Q(\g)$ is ar\-chi\-me\-de\-an,
  then $\text{Int}(\nneg{K}_d) \subset Q(\g) \cap \R[\x]_{\leq d} \subset
  \nneg{K}_d$. Thus under this assumption, \cite[Th.~1.7]{baldi2022effective} yields a degree bound $D =
  D(d,n,f,K)$ such that if $f \in \text{Int}(\nneg{K}_d)$ then $f \in Q_D(\g)$.
  Since $Q_D(\g)$ is semidefinite representable, it can be sampled through semidefinite programming:
  solving such optimization problem yields an element of the boundary of $Q_D(\g)$, thus 
  this might not be sufficient to compute an element of $\text{Int}(\nneg{K}_d)$.
  %might compute an element of the boundary
  %of $Q_D(\g)$, and hence possibly of the boundary of $\nneg{K}_d$.
  %%%it\todo{\scriptsize ``it'' == $f$?}
  %%%corresponds to the feasible set of a unique semidefinite
  %%%program\todo{\scriptsize I would be more careful}.
\end{remark}

The following Corollary shows that one can get elements of $\text{Int}(\nneg{K}_d)$ as well
through semidefinite programming, from the knowledge of a polynomial description $\g$ of $K$.

\begin{corollary}
  \label{cor:minusone}
  Let $\g$ be such that $Q(\g)$ is archimedean, and let $K=S(\g)$.
  Let $f \in \text{Int}(\nneg{K}_d)$ and $0 < \delta < f^* = \min_K f$. Then $\frac{1}{\delta} f - 1 \in
  \text{Int}(\nneg{K}_d)$ and there exist $\sigma_0^\delta,\sigma_1^\delta, \ldots,\sigma_k^\delta
  \in \Sigma_n$ such that
  $$
  \frac{1}{\delta} f - 1 = \sigma_0^\delta+\sum_i \sigma_i^\delta g_i.
  $$
\end{corollary}
\begin{proof}
  The polynomial $f-\delta$ is positive on $K$, thus by
  \Cref{lem:interior}, $f-\delta \in \text{Int}(\nneg{K}_d)$ and hence 
  ${(f-\delta)}/{\delta} = \frac{1}{\delta} f - 1 \in \text{Int}(\nneg{K}_d)$, since
  $\text{Int}(\nneg{K}_d)$ is a cone.
  Since $\frac{1}{\delta} f - 1$ is positive on $K$, we conclude by \Cref{putinar}.
\end{proof}

\Cref{cor:minusone} can be rephrased as follows: if $Q(\g)$ is archimedean and
$0 < \delta < f^* = \min_K f$, then $f/\delta \in 1+Q(\g)$. Remark that (unless
$Q(\g)$ is stable in the degree of $f$) the degrees of the SOS-multipliers for a
SOS-certificate $f/\delta \in 1+Q(\g)$ depend on $\delta$ and might be
larger than the degrees for a SOS-certificate $f \in Q(\g)$.

\subsection{Moment problem as conic feasibility}
\label{ssec:conic}

Let $C \subset V$ be a convex cone with non-empty interior, and let
$L \subset V$ be an affine space.
The \fdef{conic program} associated with $C$ and $L$ is
called \fdef{feasible} if $L \cap C \neq \emptyset$, otherwise \fdef{infeasible}.
It is called \fdef{strongly feasible} if $L \cap \text{Int}(C) \neq \emptyset$, and \fdef{weakly feasible}
if it is feasible but not strongly.

If $L$ is a linear space (that is if $0 \in L$) then
$\{0\} \subset (L \cap C)$, thus $L \cap C$ is always feasible.
If $L$ is a hyperplane, then the corresponding program is weakly feasible if and only if
$L \cap C$ is a proper face of $C$ and in this case $L$ is called a
\fdef{supporting hyperplane} for $C$: geometrically, $C$ is contained in one of the two
closed half-spaces bounded by $L$, and $L$ is tangent to the boundary of $C$.

%Let $K$ be a closed set (not necessarily semialgebraic).
%%%Let $\nneg{K} \subset \R[\x]$ be the cone of polynomials that are nonnegative on $K$: it is a closed convex pointed cone with non-empty interior. A functional $\ell \in \R[\x]^\vee$ is the moment functional of a measure $\mu$ supported on $K$ if $\ell(\x^\ba) = \int_K \x^\ba d\mu$ for all $\ba \in \N^n$. The measure is signed if it is nonnegative or nonpositive (that is $\mu(B)$ does not change sign for all Borel set $B \subset K$).

%\begin{remark}
%  Recall that being the moment functional of a measure supported on a semialgebraic set $K$ is
%  a {\it conical} property,
%  indeed if $\ell$ has this property with measure $\mu$, and $k \neq 0$, then $k \ell$ is the moment
%  functional of the measure $k\mu$. If $\mu$ is a signed measure then so is $k\mu$ and has the same sign.
%\end{remark}

%%In the following proposition, we characterize the representability of a vector $\y$ as the moments of a measure supported on a basic semialgebraic set $K$, in terms of the feasibility of some conic program.

\begin{proposition}
  \label{prop:feas}
  Let $\y = (y_\ba)_{\ba \in \N^n_d} \in \R^m$, $m=\binom{n+d}{d}$ with $\y_0>0$.
  Let $\riesz_\y \in (\R[\x]_{\leq d})^\vee$
  be the Riesz functional of $\y$, and $L_\y = \{p \in \R[\x]_{\leq d} : \riesz_\y(p)=0\}$.
  Let $K = S(\g) \subset \R^n$.
  %%%%Assume without loss of generality that $\riesz_\y(1) \geq 0$.
  The following are equivalent:
  \begin{enumerate}
  \item[$A_1$.]
    $\y \in \mathscr{M}(K)_d$;
  \item[$A_2$.]
    The conic program $L_\y \cap \nneg{K}_d$ is weakly feasible;
  \item[$A_3$.]
    There exist $\spt_1,\ldots,\allowbreak\spt_s \in K$, with $s \leq m$,
    such that $\y$ admits a representing measure $\mu$ with $\text{supp}(\mu) = \{\spt_1,\ldots,\spt_s\}$
    and $L_\y \cap \nneg{K}_d = \{p \in \nneg{K}_d : p(\spt_1) = 0, \ldots, p(\spt_s) = 0\}$.
  \end{enumerate}
  Moreover, the following are equivalent and are strong alternatives to $A_1$-$A_2$-$A_3$:
  \begin{enumerate}
  \item[$B_1$.]
    $\y \not\in \mathscr{M}(K)_d$;
  \item[$B_2$.]
    The conic program $L_\y \cap \nneg{K}_d$ is strongly feasible.
  \end{enumerate}
\end{proposition}
\begin{proof}
  The fact that $A_1$ and $B_1$ are strong alternatives is obvious, and the fact that the program
  $L_\y \cap \nneg{K}_d$ is always feasible (indeed, $L_\y$ is linear) implies that
  $A_2$ and $B_2$ are strong alternatives. Hence we only have to prove the equivalence of
  $A_1,A_2$ and $A_3$.

  {We first prove that $A_1$ is equivalent to $A_3$.}
  For $\spt \in K$, denote by $\lambda_\spt = (\spt^\ba)_{\ba \in \N^n_d} \in \mathscr{M}(K)_d$
  the sequence of moments of order $\leq d$ of the Dirac measure $\delta_{\spt}$.
  By \cite[Th.~17.2]{s17}, $\y$ admits a representing measure supported on $K$, if and only if it
  admits a representing atomic measure $\mu = \sum_{i=1}^s c_i \delta_{\spt_i}$, where
  $s \leq \dim \R[\x]_{\leq d} = \tbinom{n+d}{d}$ and for some $c_i >0$ and $\spt_1,\ldots,\spt_s
  \in K$. For every $p \in \R[\x]_{\leq d}$, one deduces
  $$
  \riesz_\y(p) = \int_K p \, d\mu = \sum_{i=1}^s c_i p(\spt_i)
  $$
  and thus $p \in L_\y$ if and only if $\sum_{i=1}^s c_i p(\spt_i) = 0$: then for
  $p \in \nneg{K}_d$, we conclude that $p$ must vanish on $\{\spt_1,\ldots,\spt_s\}$.
  {We deduce that $A_1$ and $A_3$ are equivalent}.

  {We prove now that $A_1$ and $A_2$ are equivalent.} By \Cref{tchakaloff},
  we know that $A_1$ holds if and only if $\riesz_\y$ is {non-negative} over
  $\nneg{K}_d$: this is the case if and only if the cone $\nneg{K}_d$ is contained in
  the closed half-space $L_\y^+ = \{p \in \R[\x]_{\leq d} : \riesz_\y(p) \geq 0\}$,
  and since $0 \in L_\y$, this is equivalent to $L_\y$ being a supporting hyperplane
  and the program being weakly feasible.
\end{proof}

When the conic program in \Cref{prop:feas} is weakly feasible, the set $L_\y \cap \nneg{K}_d$ is
an exposed and proper face of $\nneg{K}_d$ defined by vanishing on the finite set defined in
Item $A_3$. See also \cite{blekherman2015dimensional} and \cite[Sec.~4.4]{BPT}.
On the contrary, if the conic program is strongly feasible, we give the following definition.

\begin{definition}
  Let $\y \not\in \mathscr{M}(K)_d$. A polynomial $p \in L_\y \cap \text{Int}(\nneg{K}_d)$
  is called a \fdef{un\-re\-pre\-sen\-ta\-bi\-li\-ty certificate} for $\y$ in $K$.
\end{definition}

\Cref{cor:equiv_B} shows how to compute explicit un\-re\-pre\-sen\-ta\-bi\-li\-ty certificates
for \Cref{TMP}.

\begin{corollary}
  \label{cor:equiv_B}
  Assume that $Q(\g)$ is ar\-chi\-me\-de\-an and that conditions $B_1$-$B_2$ of \Cref{prop:feas} hold.
  %%%There exist $p,p_1,p_2 \in \text{Int}(\nneg{K}_d)$ such that
  There exists $p \in \text{Int}(\nneg{K}_d)$ such that
%  \begin{enumerate}
%  \item
    $p \in 1+Q(\g)$, $\riesz_\y(p)=0$ and $p^*>1$ is arbitrarily large.
%    \label{cor:equiv_B:item1}
%  \item
%    $\riesz_\y(p_1) < 0 < \riesz_\y(p_2)$
%    \label{cor:equiv_B:item2}
%  \end{enumerate}
  %  Moreover $p_1,p_2$ can be chosen with rational coefficients.
  %%%\todo{\scriptsize Don't we have an implicit assumption on the dimension of $K$? (maybe a stupid question).}
\end{corollary}
\begin{proof}
  Property $B_2$ of \Cref{prop:feas} ensures that there exists $f \in
  L_\y \cap \text{Int}(\nneg{K}_d)$, that is $f$ is positive over $K$ and
  $\riesz_\y(f)=0$.
  Let $0 < \delta < f^*$, and let $p = \frac{1}{\delta}f
  \in \text{Int}(\nneg{K}_d)$. From \Cref{cor:minusone} we get that $p-1 \in
  Q(\g)$, that is, $p \in 1+Q(\g)$. Moreover $\riesz_\y(p) =
  \frac{1}{\delta}\riesz_\y(f) = 0$, and $p^* = {f^*}/{\delta}>1$ is arbitrarily
  large.
  %%\todo{\scriptsize Should be more justified.},
  %%%This proves the first item.
%  Next, let $p \in \text{Int}(\nneg{K}_d)$ be such that
%  $\riesz_\y(p)=0$, and let $B \subset \text{Int}(\nneg{K}_d)$ be an open ball
%  containing $p$. Remark that both $B_- := B \cap \{q \in \R[\x]_{\leq d} :
%  \riesz_\y(q)<0\}$ and $B_+ := B \cap \{q \in \R[\x]_{\leq d} :
%  \riesz_\y(q)>0\} $ are non-empty open sets, thus both contain at
%  least one rational point.
  %%. A density argument proves the claim that $p_1$ and
  %%$p_2$ can be chosen with rational coefficients\todo{\scriptsize Not really, to
  %%be discussed.}.
%  For (3) $\Rightarrow$ (2): since $\nneg{K}$ is convex, so does $\text{Int}(\nneg{K})$, hence
%  the interval joining $p_1$ and $p_2$ is included in $\text{Int}(\nneg{K})$, and by continuity
%  of $\ell$ there is $p$ in this interval where $\ell$ vanishes.
\end{proof}

\begin{remark}
  If $\y \in \Q^m$, then $p$ in \Cref{cor:equiv_B} can be chosen with rational coefficients.
  Indeed, the hyperplane $L_\y$ is defined by an equation with rational
  coefficients, and $L_\y \cap \text{Int}(\nneg{K}_d)$ is a non-empty open subset of
  $L_\y$, hence it contains a rational point $f$. Choosing $\delta \in \Q$ in the proof of
  \Cref{cor:equiv_B} is thus sufficient to get a rational certificate.

  Nevertheless, let us recall from \cite{scheiderer2016sums}
  that there exist polynomials in $\Q[\x]$, that are sums of squares as elements of $\R[\x]$ but not
  as elements of $\Q[\x]$. In our context, this means that the rational unrepresentability
  certificate $p$ might not admit rational certificates of positivity showing that $p \in 1+Q(\g)$
  (see also \cite{naldi2021conic} for the existence of rational certificates in conic programming).
\end{remark}

  Any polynomial $p$ as in \Cref{cor:equiv_B} is such that
  $p-1 \in Q(\g)$. In particular, there exists $D>0$ such that $p-1 \in Q(\g)[D]$, that is
  $p \in 1+Q(\g)[D]$. Bounds for $D$ are given, {\it e.g.}, in \cite[Th.~1.7]{baldi2022effective}.
  Below we give a bound on $D$ as a function of the input of \Cref{algo:naif}, see
  \Cref{unrepr_deg} and \Cref{bound_unrepr_deg}.

%\begin{remark}
%  \Cref{cor:equiv_B} ensures that, when $\y \in \Q^m$ is not the vector of partial moments of a measure
%  supported on $K=S(\g)$, then there exists a rational unrepresentability certificate
%  $p \in \text{Int}(\nneg{K}_d)$ for $\y$.
%  Nevertheless, let us recall from \cite{scheiderer2016sums}
%  that there exist polynomials in $\Q[\x]$, that are sums of squares as elements of $\R[\x]$ but not
%  as elements of $\Q[\x]$. In our context, this means that the rational unrepresentability
%  certificate $p$ might not admit rational certificates of positivity showing that $p \in 1+Q(\g)$
%  (see also \cite{naldi2021conic} for the existence of rational certificates in conic programming).
%%  Below we are going to use representations of elements in $\text{Int}(\nneg{K}_d)$ through weighted sums of
%%  squares in the spirit of \Cref{putinar}, and the result of Scheiderer shows that there might exist
%%  rational certificates of unrepresentability of
%%  in the sense of \Cref{cor:equiv_B} might exist in $\Q[x]$, but might not admit a rational sums of squares
%%  certificate.\todo{\scriptsize Not very clear yet}\footnote{In other words: there might be examples of vectors $\y$
%%    which are not in $\mathscr{M}(K)_d$, and thus that admit a certificate of unrepresentability $p \in 1+Q(\g)$
%%    with rational coefficients, but such that $p$ does not admit a rational sum of squares certificate. My
%%    {conjecture} is that all Scheiderer polynomials are unrepresentability certificates of some vector
%%  $\y$.}.
%\end{remark}

The following example shows that the certificate of \Cref{cor:equiv_B}
might exist in non-archimedean contexts.

\begin{example}
  \label{non_archimedean_certificate}
  The vector $\y = (1,1,0)$ is not a univariate moment vector (indeed the moment matrix
  $M_1(\y) = \left(\begin{smallmatrix} 1 & 1 \\ 1 & 0\end{smallmatrix}\right)$ is not
    positive semidefinite). Remark that
  $\R = S(0)$ and $Q(0) = \Sigma_n$ is not archimedean, but an unrepresentability certificate
  in the spirit of \Cref{cor:equiv_B} exists. Indeed $\riesz_\y(p_0+p_1x+p_2x^2) = p_0+p_1$ and
  the following identity holds for all $p \in L_\y$:
  $$
  p_0-p_0x+p_2x^2 =
  1+
  \begin{pmatrix}
    1 & x
  \end{pmatrix}
  \begin{pmatrix}
    p_0-1 & -\frac{p_0}{2} \\
    -\frac{p_0}{2} & p_2
  \end{pmatrix}
  \begin{pmatrix}
    1 \\ x
  \end{pmatrix}.
  $$
  The identity is a SOS-certificate for $p \in 1+Q(\g)$ if and only if the
  Gram matrix on the right hand side is positive semidefinite. This yields a spectrahedral
  representation of the set of unrepresentability certificates of $\y$:
  $$
  \left\{
  p \in \R[\x]_{\leq 2} :
  p_1=-p_0, 
  \left(\begin{smallmatrix}
    p_0-1 & -\frac{p_0}{2} \\
    -\frac{p_0}{2} & p_2
  \end{smallmatrix}
  \right) \succeq 0
  \right\}
  $$
  For instance the polynomial $p = 2-2x+x^2 = 1+(1-x)^2 \in (1+Q(0)) \cap L_\y$.

  We remark that the existence of such $p$ implies $\y \not\in \mathscr{M}(\R)_{2}$.
  Indeed, if $\y \in \mathscr{M}(\R)_{2}$,
  then by $A_3$ of \Cref{prop:feas}, there would exist a representing measure with finite
  support, thus one would have $\y \in \mathscr{M}([-R,R])_{2}$ for some $R>0$.
  Now, since $1 + Q(0) \subset 1+Q(R^2-x^2)$ for every $R>0$, we deduce that $p \in 1+Q(R^2-x^2)$,
  thus $p \in \text{Int}(\nneg{[-R,R]}_2)$ (according to \Cref{lem:interior}) and thus
  $\y \not\in \mathscr{M}([-R,R])_{2}$, for every $R>0$ (by $B_2$, \Cref{prop:feas}).
  \hfill$\Box$
\end{example}

The polynomial $p$ in \Cref{non_archimedean_certificate}, together with its positivity
certificate $1+(1-x)^2$, allows to check rigourously the un\-re\-pre\-sen\-ta\-bi\-li\-ty
of $\y=(1,1,0)$. Nevertheless, in general the archimedianity hypothesis cannot be
dropped, as shown in \Cref{absence_certificate_non_archimedean}.

\begin{example}
  \label{absence_certificate_non_archimedean}
  Let $\y = (1,0,0,0,1)$ be the vector of \Cref{example_not_closed}, with $\g=0$.
  The semidefinite program in \Cref{cor:equiv_B} is
  infeasible: indeed $Q(0)=\Sigma_{n}$ and it is easy to check
  that there is no polynomial $p = \sum_{i=0}^4p_ix^i$ satisfying
  \begin{align*}
    \riesz_\y(p) = p_0 + p_4 & = 0 \\
    p & = 1 +
    \begin{pmatrix}
      1 & x & x^2
    \end{pmatrix}
    X
%    \begin{pmatrix}
%      x_{11} & x_{12} & x_{13} \\
%      x_{12} & x_{22} & x_{23} \\
%      x_{13} & x_{23} & x_{33}
%    \end{pmatrix}
    \begin{pmatrix}
      1 & x & x^2
    \end{pmatrix}^T \\
    X & \succeq 0.
%    \begin{pmatrix}
%      1 \\ x \\ x^2
%    \end{pmatrix}
  \end{align*}
  Indeed, the constraints imply that $p_4 = -p_0 = -(1+X_{11}) < 0$, thus $p$ is
  negative at infinity, in particular, $p\not\in\nneg{\R}$.

  Nevertheless for every
  $R>0$, with $\g=(R-x,R+x)$, \Cref{cor:equiv_B} ensures that there exists
  $p_R \in 1+Q(R-x,R+x)$
  such that $\riesz_\y(p_R)=0$, that certifies that $\y \not\in \mathscr{M}([-R,R])_4$:
  the polynomial $p_R$ is any solution of the following parametric linear matrix inequality
  with $\sigma_0,\sigma_1,\sigma_2 \in \Sigma_{1}$:
  $$
  p_0+p_1x+p_2x^2+p_3x^3-p_0x^4 = 1+\sigma_0+\sigma_1(R-x) + \sigma_2(R+x).
  $$
  The fact that $\g=(R-x,R+x)$ is the natural description of $[-R,R]$ (see \cite[Sec.~2.7]{marshall2008positive})
  implies that $Q(\g)$ is stable, with stability function $D(d)=d$ (see {\it e.g.} \cite[Prop.~3.3]{s17}), and hence one can
  assume $\sigma_0 \in \Sigma_{1,4}$ and $\sigma_1,\sigma_2 \in \Sigma_{1,2}$.
  \hfill$\Box$
\end{example}

We terminate this series of examples with a bivariate one.

\begin{example}
  \label{example_didier}
  Let $n=2,d=6$ and let $\y = (y_{\ba})_{\ba \in \N^2_6}$ be the vector in $\R^{28}$ whose non-zero
  entries are
  \begin{align*}
    y_{00} & = 32          & y_{22} & = 30 \\
    y_{20} & = y_{02} = 34 & y_{60} & = y_{06} = 128 \\
    y_{40} & = y_{04} = 43 & y_{42} & = y_{24} = 28.
  \end{align*}
  We claim that there is no nonnegative Borel measure supported on the unit ball
  $K = \{\pt=(a_1,a_2) \in \R^2 : a_1^2+a_2^2 \leq 1\}$ whose moments up to degree $6$ agree with
  $\y$. Remark that the semialgebraic set $K = S(1-x_1^2-x_2^2)$ is compact and $Q(1-x_1^2-x_2^2)$
  is archimedean, in particular it is not stable. We give below in \Cref{example_didier_suite}
  an unrepresentability certificate of small degree certifying that $\y \not\in \mathscr{M}(K)_d$,
  proving our claim.
  \hfill$\Box$
\end{example}

%\begin{lemma}
 % \label{lem:stron_feas}
  %Let $\ell \in \R[\x]^\vee$ be a non-zero linear functional, and let $L =  \{p \in \R[\x] : \ell(p) = 0\}$.
%  The following are equivalent
%  \begin{enumerate}
%  \item
%    Neither $\ell$ nor $-\ell$ is the moment functional associated with any measure $\mu$ supported on $K$;
%  \item
%    $\nneg{K} \cap L$ is strongly feasible.
%  \end{enumerate}
%\end{lemma}
%\begin{proof}
%%  Assume (1), then appying \Cref{prop:feas} one deduces that $\nneg{K} \cap L$ is not weakly feasible, thus it must be strongly feasible (since $L$ is linear). If it is strongly feasible,  it means that there exists $q \in \text{Int}(\nneg{K})$ such that $\ell(q)<0$, hence by  \cite[Th.~1.12]{s17}
%\end{proof}

\subsection{SDP-based algorithm}
\label{ssec:naif_algo}

The results of \Cref{ssec:conic} yield the following alternatives
for \Cref{TMP}. Given $\y = (\y_\ba)_{\ba \in \N^n_d}$ and $K=S(\g)$, according to \Cref{prop:feas}:
\begin{itemize}
\item
  either $\y \not\in \mathscr{M}(K)_d$, in which case a certificate of un\-re\-pre\-sen\-ta\-bi\-li\-ty
  is given by a polynomial $p \in \text{Int}(\nneg{K}_d)$ such that $\riesz_\y(p)=0$ (\Cref{cor:equiv_B})
\item
  or $\y \in \mathscr{M}(K)_d$, in which case there exists an atomic measure $\mu = \sum_{i=1}^s c_i
  \delta_{\spt_i}$ representing $\y$ (Property $A_3$ in \Cref{prop:feas}).
%  \stodo{spécifier le certificat qu'on choisit: soit une mesure
%    atomique $\mu = \sum_i c_i \delta_{\spt_i}$ (voir \Cref{repr_system}) ou une extension plâte de $\y$
%  (voir \Cref{repr_flat})}
\end{itemize}

We describe our main algorithm. \\

\fbox{
  \parbox[center]{\textwidth}{
\begin{algorithm}
  \algoname{certify\_moment}
  \label{algo:naif}
  \begin{algorithmic}[1]
    \Require{
      \Statex \textbullet~ $n,d \in \N$
      \Statex \textbullet~ A vector $\y \in \R^m$, with $m=\tbinom{n+d}{d}$
      \Statex \textbullet~ Polynomials $\g = (g_1,\ldots,g_k) \in \R[\x]^k$
%      \Statex \textbullet~ A threshold $T \in \N$
      \Statex \textbullet~ A threshold $D \in \N$
     }
    \Ensure{
      \Statex \textbullet~ Either $(p,\Sigma)$ where $p \in \R[\x]_{\leq d}$ satisfies \Cref{cor:equiv_B},
      and $\Sigma \in \Sigma_n^{k+1}$ is a certificate for $p \in 1+Q(\g)[D]$
      \Statex \textbullet~ or a measure $\mu = \sum c_i \delta_{\spt_i}$ satisfying $A_3$ in \Cref{prop:feas}
      %%an output message ``$\y$ is a moment vector''
    }
%    \State $D \gets 1$
%    \While{$D \leq T$}
    \State $(p,\Sigma) \gets {\bf find\_certificate}(n,d,\y,\g,D)$ \label{step_sdp}
    \InlineIf{$p \neq []$}{\Return $(p,\Sigma)$} \label{if_return}
%    \State $D \gets D+1$
%    \EndWhile
    \State \Return ${\bf find\_measure}(n,d,\y,\g)$ \label{step:extract:meas}
    %%``$\y$ is a moment vector''
  \end{algorithmic}
\end{algorithm}
}}

\vspace{0.3cm}

\Cref{algo:naif} depends on two subroutines. The first one, {\bf fi\-nd\_cer\-ti\-fi\-ca\-te} at Step \ref{step_sdp},
returns, if it exists, an un\-re\-pre\-sen\-ta\-bi\-li\-ty certificate $p \in \R[\x]_{\leq d}$ for $\y$, together with
a SOS-certificate for $p$ as element of $1+Q(\g)[D]$.

\vspace{0.3cm}

\fbox{
  \parbox[left]{\textwidth}{
\begin{algorithm}
  \algoname{find\_certificate}
  \label{algo:find:certificate}
  \begin{algorithmic}[1]
%    \Require{
%      \Statex \textbullet~ $n,d \in \N, \y \in \R^m$ and $\g \in (\R[\x])^k$ ({\it cf.} \Cref{algo:naif})
%      \Statex \textbullet~ $D \in \N$
%     }
%    \Ensure{$(p,\Sigma)$ where $p \in \R[\x]_{\leq d}$ satisfies Item 1 of \Cref{cor:equiv_B},
%      and $\Sigma \in \Sigma_n^{k+1}$ is a certificate for $p \in 1+Q(\g)$
%    }
    \State $p \gets []$, $\Sigma \gets []$, $g_0 \gets 1$
    \State Find $p \in \R[\x]_{\leq d}$ and $\sigma_0,\sigma_1,\ldots,\sigma_k \in \Sigma_n$ such that
    \begin{itemize}
    \item
      $\riesz_\y(p)=0$
    \item
      $p = 1+\sum_{i=0}^k \sigma_i g_i$, $\deg(\sigma_i g_i) \leq D$ \label{algo:find:certificate:sdp}
    \end{itemize}
    %    \State $D \gets 1$
    %    \While{$D \leq T$}
    %    \State Find $p \in (1+Q(\g)) \cap \R[\x]_{\leq d}$ s.t. $\riesz_\y(p)=0$ \label{step_sdp}
    %    \InlineIf{$p$ exists at Step \ref{step_sdp}}{\Return $p$} \label{if_return}
    %    \State $D \gets D+1$
    %    \EndWhile
    \State $\Sigma \gets [\sigma_0, \sigma_1, \ldots, \sigma_k]$
    \State \Return $(p, \Sigma)$
  \end{algorithmic}
\end{algorithm}
}}

\vspace{0.3cm}

\Cref{algo:find:certificate} can be performed by solving one (finite-dimensional)
SDP feasibility program whose unknowns are $p, \sigma_0, \allowbreak\sigma_1, \ldots, \sigma_k$.
First, the constraint $\riesz_\y(p)=0$ is linear in $p$. Next, denoting
$\delta_i=\lfloor (D-\deg(g_i))/2 \rfloor$, the constraints $\sigma_i \in \Sigma_n$ and
$\deg(\sigma_i g_i) \leq D$ are equivalent to the existence of a symmetric matrix $X_i \succeq 0$,
of size $\binom{\delta_i+n}{n}$, such that $\sigma_i = v^T X_i v$, where $v$ is a linear basis of
$\R[\x]_{\delta_i}$.
Finally the constraint $p=1+Q(\g)[D]$ in Step \ref{algo:find:certificate:sdp} is affine linear in $p$
and in the entries of $X_0, X_1, \ldots, X_k$.
We give upper bounds for the value of $D$ in \Cref{bound_unrepr_deg}.

The second routine, {\bf find\_measure}, is called if and only if
\Cref{algo:naif} reaches Step \ref{step:extract:meas}. It returns a $s$-atomic
measure representing the vector $\y$, for some $s \leq \binom{n+d}{d}$.

\vspace{0.3cm}

\fbox{
  \parbox[left]{\textwidth}{
\begin{algorithm}
  \algoname{find\_measure}
  \label{algo:extract:measure}
  \begin{algorithmic}[1]
%    \Require{
%      \Statex \textbullet~ $n,d \in \N, \y \in \R^m$ and $\g \in \R[\x]^k$ ({\it cf.} \Cref{algo:naif})
%     }
%    \Ensure{
%      \Statex \textbullet~ A measure $\mu = \sum c_i \delta_{\spt_i}$ satisfying $A_3$ in \Cref{prop:feas} 
%    }
    \State $s=1$
    \While{$s \leq \binom{n+d}{d}$}
    \State Find $\c \in \R^s$ and $\bm{U} = (\spt_1,\ldots,\spt_s) \in (\R^n)^s$ s.t.
    \begin{itemize}
      \item $\spt_1,\ldots,\spt_s \in S(\g)$
      \item $c_1\spt_1^\ba+\cdots+c_s\spt_s^\ba = y_{\ba}$ for all $\ba \in \N^n_d$
    \end{itemize}
    \InlineIf{solution exists}{\Return $(\c,\bm{U})$}
    \State $s \gets s+1$
    \EndWhile
%    \State \Return ``$\y$ is a moment vector''
  \end{algorithmic}
\end{algorithm}
}}

\vspace{0.3cm}

The routine {\bf find\_measure} can be performed by existing algorithms computing one point per
connected component of basic semialgebraic sets applied to the set of elements $(\c,\bm{U}) \in \R^s \times (\R^n)^s$
satisfying the inequalities defining $K$ and the polynomial equations $\sum_i c_i \spt_i^\ba = y_\ba$,
$\ba \in \N^n_d$. The equations have a multivariate Vandermonde structure. A precise complexity analysis
of {\bf find\_measure} is left to future work.

We define now an integer function of the input of \Cref{TMP}.

\begin{definition}
  \label{unrepr_deg}
  Let $K=S(\g)$ and $\y \not\in \mathscr{M}(K)_d$. The \emph{unrepresentability degree} of $\y$ in $K$
  is the minimum integer $D=D(n,d,\y,\g)$ such that there exists $p \in 1+Q(\g)[D]$ satisfying $\riesz_\y(p)=0$.
\end{definition}

The unrepresentability degree is well defined, according to \Cref{cor:equiv_B}.
We prove the correctness of \Cref{algo:naif}.

\begin{theorem}[Correctness]
  \label{correctness}
  Let $\y = (y_\ba)_{\ba \in \N^n_d}$, and let $\g=(g_1,\ldots,g_k) \in \R[\x]^k$ be such that
  $Q(\g)$ is archimedean.
  There exists $D = D(n,d,\y,\g) \in \N$ such that \Cref{algo:naif} with input
  $(n,d,\y,\g,D)$ terminates and is correct.
\end{theorem}
\begin{proof}
  Let $K=S(\g)$. Since $Q(\g)$ is archimedean, $K$ is compact and by \Cref{lem:interior},
  $\text{Int}(\nneg{K}_d) = \{p \in \R[\x]_{\leq d} : p(\pt) > 0, \,\forall\,\pt\in K\}$.
  By \Cref{putinar}, if a polynomial is positive on $K$, then it belongs to $Q(\g)$.
  We deduce that $\text{Int}(\nneg{K}_d) \subset Q(\g) \cap \R[\x]_{\leq d} \subset
  \nneg{K}_d$.

  We claim that for $D \in \N$ large enough, then $\y \not\in \mathscr{M}(K)_d$
  if and only if {\bf certify\_moment} returns a polynomial $p$ at Step \ref{if_return}, that is, if and
  only if the semidefinite program at Step \ref{step_sdp} is feasible.

  Assume $\y \not\in \mathscr{M}(K)_d$ and let $D$ the unrepresentability degree of
  $\y$ in $K$.
  By \Cref{cor:equiv_B}, 
  there exists a polynomial $p \in 1+Q(\g)[D] \subset \text{Int}(\nneg{K}_d)$, such that
  $\riesz_\y(p)=0$. In other words, {\bf find\_certificate} returns $(p,\Sigma)$ with $p \neq []$,
  and hence \Cref{algo:naif} returns its output at Step \ref{if_return}.

  For the reverse implication, suppose that the semidefinite program at Step \ref{step_sdp} is feasible
  for some degree $D$. Let $p \in \R[\x]_{\leq d}$ be a solution of such program. Then $p = 1+q$
  for some $q \in Q(\g)[D]$, in particular, $p^* \geq 1$ on $K$, that is, $p$ is positive on $K$.
  Since $\riesz_\y(p)=0$, and by compactness of $K$, one has $p \in \text{Int}(\nneg{K}) \cap L_\y$
  and again applying \Cref{prop:feas} one concludes $\y\not\in\mathscr{M}(K)_d$.
  This proves the claim.

  Finally, remark that this also shows that $\y \in \mathscr{M}(K)_d$ if and only if {\bf certify\_moment}
  reaches Step \ref{step:extract:meas}. If this is the case, {\bf find\_measure} computes the support
  and the weights of an $s$-atomic measure representing $\y$: such measure exists for some $s \leq \binom{n+d}{d}$,
  according to \Cref{prop:feas}.
\end{proof}

\section{Bound on the unrepresentability degree}
\label{bound_unrepr_deg}
{\it A priori} bounds on the degree of Putinar certificates that only depend on the
degree of the polynomial exist for stable quadratic modules. Nevertheless,
stability and ar\-chi\-me\-de\-ani\-ty properties are mutually exclusive in
dimension $n \geq 2$. 

In this section we give general bounds for the
unrepresentability degree of a vector $\y \not\in \mathscr{M}({K})_d$ in a
compact basic semialgebraic set $K = S(\g)$. A key ingredient to do this is the
use of already existing quantitative analysis of computer algebra algorithms
performing quantifier elimination over the reals.

We first recall the following bound on the degree of a Putinar representation of
a polynomial in $\text{Int}(\nneg{K}_d)$, for an archimedean quadratic module
$Q(\g)$, given in \cite{baldi2022effective}. Let $f \in \text{Int}(\nneg{K}_d)$.
We denote by $\epsilon(f) := {f^*}/{\|f\|}$ where
$$
f^*=\min_{\pt \in K} f(\pt)
\,\,\,\,\,\,\,\,\,
\text{  and  }
\,\,\,\,\,\,\,\,\,
\|f\| = \max_{\pt \in [-1,1]^{n}}f(\pt).
$$
 % is
% a norm on the vector space $\R[\x]_{\leq d}$
Under the following assumptions:
\begin{itemize}
\item
  $1-\sum_i x_i^2 \in Q(\g)$
\item
  $\|g_i\| \leq \frac12$ for all $i=1,\ldots,k$
\end{itemize}
then by \cite[Th.~1.7]{baldi2022effective} there exists a function $\gamma = \gamma(n,\g)$
such that $f \in Q(\g)[D]$ for $D$ of the order of
\begin{equation}
  \label{boundBM}
\gamma(n,\g) \, d^{3.5 \L n} \, \epsilon(f)^{-2.5 \L n}
\end{equation}
where $\mathfrak{c}$ and $\L$ are the $\L$ojasiewicz coefficients as they are
defined in~\cite[Def.~2.4]{baldi2022effective}. \\

Let now $(n,d,\y,\g)$ be the input of \Cref{algo:naif}. We assume in the whole
section that $\y \in \Q^N$, with $N =\binom{n+d}{d}$, that the basic semialgebraic set $K =
S(\g) \subset \R^{n}$ is defined by polynomial inequalities $\g=(g_1,\ldots,g_k)
\subset \Q[\x]^k$ and that $Q(\g)$ is archimedean.

Assume $\y \not\in \mathscr{M}(K)_d$, and let $p \in 1+Q(\g)$ satisfy
$\riesz_\y(p)=0$ and $p^* = 1+m$ for some arbitrary constant $m>0$.
Such unrepresentability certificate exists by
\Cref{cor:equiv_B}. From \eqref{boundBM}, one has that $p \in 1+Q(\g)[D]$ with
$D$ depending on
$$
\epsilon(p-1) = \frac{p^*-1}{\|p-1\|} \geq \frac{1}{1+\|p\|}
$$
where the last inequality derives from \Cref{cor:equiv_B} choosing without loss
of generality $m=1$.
In the following, we provide an upper bound $B$ on $\|p\|$: this will yield a lower
bound on $\epsilon(p-1)$, hence an upper bound on $D$.

As already said, in order to do that, we consider the formulation in terms of
quantifier elimination over the reals that solves the truncated moment problem
in the sense of \Cref{prop:feas} and \Cref{cor:equiv_B}, and we use quantitative results
on quantifier elimination over the reals from \cite{BPR}.

Consider the following formula with quantified variables $\x=(x_1, \allowbreak\ldots, x_n)$ and
parameters the unknown coefficients $\{p_\ba : \ba \in \N^n_d\}$ of a polynomial $p \in \R^N$:
%%% MOHAB'S FORMULA
%\begin{equation}
%  \label{eq:qe}
%  \forall \x \in \R^{n} \quad
%  \left(\bigwedge_{i=1}^k g_i(\x)\geq 0\right) \wedge \riesz_\y(p) = 0 \Rightarrow p(\x) \geq 2.
%\end{equation}
%%% SIMONE'S FORMULA
\begin{equation}
  \label{eq:qe}
  \riesz_\y(p) = 0
  \,\,\, \wedge \,\,\,
  \left(\forall \x \in \R^{n} \quad \bigwedge_{i=1}^k g_i(\x)\geq 0 \Rightarrow p(\x) > 1 \right).
\end{equation}

\begin{lemma}
  \label{lem:formula4}
  Let $S \subset \R^N$ be the semialgebraic set defined by the quantifier-free formula obtained from
  \eqref{eq:qe} after eliminating the quantified variables. Then $S$ is an open subset of $L_\y$ with
  respect to the induced topology.
\end{lemma}
\begin{proof}
  For $p \in \R[\x]_{\leq d}$ and $A \subset \R[\x]_{\leq d}$, we denote by $p+A:=\{p+q : q \in A\}$.
  Since $Q(\g)$ is archimedean, $K$ is compact and according to \Cref{lem:interior}, $\text{Int}(\nneg{K}_d)$
  consists exactly of polynomials that are positive on $K$. Thus $S = L_\y \cap (1+\text{Int}(\nneg{K}_d))$.
  Moreover by \Cref{prop:feas}, we know that $L_\y \cap \text{Int}(\nneg{K}_d)$ is non-empty and open in $L_\y$.
  Now since $L_\y$ is a linear space intersecting $\text{Int}(\nneg{K}_d)$, which is an open convex cone, so
  does the affine space $-1+L_\y$. Thus $T = (-1+L_\y) \cap \text{Int}(\nneg{K}_d)$ is open in $-1+L_\y$ and hence
  $S = 1+T$ is open in $1+(-1+L_\y)=L_\y$.
\end{proof}

Observe that the constraint $\riesz_\y(p)=0$ in \eqref{eq:qe} is linear in $p$ and $\y \neq 0$.
%%%Thus the semialgebraic set obtained by eliminating the quantified variables $\x$ from \eqref{eq:qe} is exactly the scaled image of the set $S = L_\y \cap \text{Int}(\nneg{K}_d)$ of \Cref{cor:equiv_B}, which is non-empty and open in $L_\y = \R^{N-1}$.
Thus one of the coefficients of $p$, say $p_{\ba'}$, can be eliminated, yielding
an formulation of the quantifier elimination problem \eqref{eq:qe}:
%%% MOHAB
%\begin{equation}
%  \label{eq:qe2}
%  \forall \x \in \R^{n} \quad
%  \left(\bigwedge_{i=1}^k g_i(\x)\geq 0\right) \Rightarrow \tilde{p}(\x) \geq 2
%\end{equation}
%%% SIMONE
\begin{equation}
  \label{eq:qe2}
  \forall \x \in \R^{n} \quad \bigwedge_{i=1}^k g_i(\x)\geq 0 \Rightarrow \tilde{p}(\x)>1
\end{equation}
where $\tilde{p}$ is the polynomial obtained when substituting $p_{\ba'}$ in $p$
by a linear form in the other coefficients $p_\ba$ using $\riesz_\y(p) = 0$.
Below we abuse of notation and consider the set $S$ defined in \Cref{lem:formula4} and by the
previous formulae as embedded in $L_\y$ identified with $\R^{N-1}$.
We conclude that the set $S$ has non-empty interior in $\R^{N-1}$.

%We let $\tilde{S}\subset \R^{N-1}$ be the semialgebraic set defined by the above quantified formula~\eqref{eq:qe2}. This is the set of polynomials in $\R[\x]_{\leq d}$ which cancels $\riesz_\y$ and takes values greater than or equal to $2$ over $K$.
%%%%%Observe also that~\eqref{eq:qe} implies that for any $\tilde{p}\in \tilde{S}$, then the minimum of $\tilde{p}-1$ over $K$ is greater than $1$.
We denote by
$$
\gdeg = \max_i \big\{\deg(g_i), i=1,\ldots,k\big\}
$$
and by $\tau_\g$ the maximum bit size of the coefficients of the $g_i$'s. Note that we can
multiply in~\eqref{eq:qe2} the polynomials $g_i$ by the (positive) least commun
multiple of the denominators of their coefficients to obtain equivalent
inequalities but with coefficients in $\Z$. These least common multiples have
height bounded by $(\binom{n+d_\g}{n} + 1)\tau_\g$.

Further, we denote by $\tau_\y$ the maximum bit size of the coefficients of
$\y$. As above, the equation $\riesz_\y(p) = 0$ can be rewritten with
coefficients in $\Z$ of bit size bounded by $(N+1)\tau_\y$. We set
\begin{equation}
  \label{tau}
\tau = \max\left((N+1)\tau_\y, \left(\binom{n+d_\g}{n} \right) \tau_\g\right).
\end{equation}
Note that $\tau$ is a bound on the integer coefficients of
the polynomial constraints in \eqref{eq:qe2} once we have multiplied each of
them by the least common multiple of the denominators of their coefficients.

Finally, let $\delta = \max(d + 1, d_\g)$. Note that $\delta$ dominates the
maximum degree of the polynomial constraints in~\eqref{eq:qe2}, indeed, the polynomial
$\tilde{p} \in \R[\x,p_\ba : \ba \in \N^n_d]$ has degree $d$ in $\x$ and has degree
$1$ with respect to its unknown coefficients.

% The vector space $\R[\x]_{\leq d}$ is identified to a real vector
% space of dimension $N$ endowed with variables $\{p_\ba : \ba \in \N^n_d\}$.
% For such a polynomial $p$, we then reuse the notation $\riesz_\y(p)$ for the Riesz functional associated to $\y$.

%%%% Je ne comprends pas.. c'est plutot p^*>=2 car p>=2 sur K
%Note that, for any $p\in S$, $p^\star = \inf_{\x\in K}p(\x)\geq 1$. Hence, we
%need now to find a bound $B\in \R$ such that $S$ contains a polynomial $f$ such
%that $\|f\|\leq B$ which implies that $\epsilon(f) \geq \frac{1}{B}$.

\begin{proposition}
  \label{prop:bound_eps}
  There exists $\tilde{p} \in {S}$ with $\|\tilde{p}\| \leq B$ for $B$ in
  $$
  \tau^{O(1)} \left(k(\delta+1)\right)^{O(n(N-1))}.
  $$
%%%  Reusing the notation introduced above, we can take $B$ in $\tau^{O(1)} \left(s(d+1)\right)^{O(nN)}$. 
\end{proposition}
\begin{proof}
  We start by providing some quantitative bounds on the quantifier-free formula
  which defines ${S} \subset \R^{N-1}$ obtained by eliminating the quantified
  variables $\x=(x_1,\ldots,x_n)$ in \eqref{eq:qe2}. By \cite[Theorem 14.16]{BPR}
  such a formula satisfies the following properties
  \begin{itemize}
  \item It can be obtained with polynomials of degree lying in $(\delta+1)^{O(n)}$;
  \item the bit size of the coefficients of these polynomials lies in $\tau
    (\delta+1)^{O(n(N-1))}$;
  \item this formula is a disjunction of $k^{n+1}\delta^{O(n)}$ conjunctions of
    $k^{n+1}\delta^{O(n)}$ disjunctive formulas of polynomial inequalities involving
    $k^{n+1}\delta^{O(n)}$ polynomials.
  \end{itemize}

  Recall that since $S$ is open, it coincides with its interior $\text{Int}(S)$.
  We aim at computing one point with rational coordinates in $S$. To do this, we
  just put these disjunctions in closed form, replace non-strict inequalities by
  strict inequalities and call an algorithm for computing at least one point
  with rational coordinates in $S$; see e.g.~\cite[Theorem 4.1.2]{BPR96}.

  Note that the input to such an algorithm is a system of polynomial strict
  inequalities in $\R[p_\ba : \ba \in \N^n_d]$ of degree $(\delta+1)^{O(n)}$ with bit size
  coefficients in $\tau (\delta+1)^{O(n(N-1))}$. By~\cite[Theorem 4.1.2]{BPR96}, if the
  semialgebraic set defined by the input is non-empty, then it outputs a point
  with rational coordinates of bit size bounded by $\tau^{O(1)} \left( k
    (\delta+1)\right)^{O(n(N-1))}$.

  All in all, this bounds the bit size of the coefficients of some polynomial
  $\tilde{p} \in S$ (with rational coordinates). Since the number of
  these coefficients is $N-1$, the $2$-norm of $\tilde{p}$ still lies in
  \[\tau^{O(1)}  \left( k (\delta+1)\right)^{O(n(N-1))}.\]
  Finally, observe that $\|\tilde{p}\|=\min_{\pt \in [-1,1]^{n}}\tilde{p}(\pt)$
  is bounded above by the $2$-norm of $\tilde{p}$.
\end{proof}
\begin{corollary}
  \label{bound_pessimistic}
  Let $\tau_\y$ and $\tau_\g$ bound the bit-size of $\y$ and $\g$, respectively,
  and let $d_\g$ be a bound on the degrees of $g_1,\ldots,g_k$. Let $\tau$ be as in
  \eqref{tau} and $\delta=\max\{d+1,d_\g\}$.
  If $\y \not\in \mathscr{M}(K)_d$, then the degree of unrepresentability of $\y$
  in $K$ is in
  $$
  \gamma(n,\g) \, d^{3.5 \L n} \, \tau^{O(\L n)} \, \left({k}(\delta+1)\right)^{O(\L n^2(N-1))}.
  $$
\end{corollary}
\begin{proof}
  With the notation introduced in \eqref{boundBM}, by \Cref{cor:equiv_B}, there exists
  $p \in 1+Q(\g)[D]$ and by applying \cite[Th.~1.7]{baldi2022effective} and
  \Cref{prop:bound_eps}, the degree $D$ is bounded above by
  \begin{align*}
    & \gamma(n,\g) \, d^{3.5 \L n} \, \epsilon(p-1)^{-2.5 \L n} \\
    & \leq \gamma(n,\g) \, d^{3.5 \L n} \, (1+\|p\|)^{2.5 \L n} \\
    & \leq \gamma(n,\g) \, d^{3.5 \L n} \, (\tau^{O(1)} \left({k}(\delta+1)\right)^{O(n(N-1))})^{2.5 \L n} \\
    & = \gamma(n,\g) \, d^{3.5 \L n} \, \tau^{O(\L n)} \, \left({k}({\delta}+1)\right)^{O(\L n^2(N-1))}
  \end{align*}
\end{proof}

We terminate with a bivariate example showing that the bound of \Cref{bound_pessimistic}
is usually quite pessimistic.

\begin{example}[{\Cref{example_didier}} continued]
  \label{example_didier_suite}
Let $\y \in \R^{28}$ be the vector defined in \Cref{example_didier}.
Consider the polynomial
$$
p = 1+\frac89 (1-x_1^2-x_2^2)
$$
One checks that $p \in 1+Q(1-x_1^2-x_2^2)$ and $\riesz_\y(p)= y_{00}(1+\frac89)-\frac89(y_{20}+y_{02})=0$.
Since $K=S(1-x_1^2-x_2^2)$ is compact, $p$ certifies that the conic program defined in
\Cref{prop:feas} is strongly feasible, in other words, that $\y \not\in \mathscr{M}(K)_d$.
  \hfill$\Box$
\end{example}

\section{Conclusions and perspectives}

The goal of this work was to undertake a systematic analysis of the computational complexity of the truncated moment problem on semialgebraic sets. Preliminary results concern the existence of algebraic certificates for vectors that are not representable as moments of measures, and upper bounds on the degree of SOS representations of these certificates.

Our contribution offers several challenges and research directions in the computational aspects of the truncated moment problem, let us mention a few. One of these is the need of efficient algorithms for classes of polynomial systems with Vandermonde structure as that defined in the routine {\bf find\_measure}. 

A second one is to refine the quantifier-elimination bound given in \Cref{bound_pessimistic}. Unlike the viewpoint of the so-called effective Putinar Positivstellensatz introduced in \cite{baldi2022effective}, for which degree bounds depend on the polynomial itself, it is clear from our analysis that for the complexity analysis of the TMP one needs to give uniform degree bounds that only depend on the input of the TMP. One way of getting such uniform bounds is to consider manifestly positive polynomials such as those in $1+Q(\g)$ for compact $K=S(\g)$.

A final perspective is to extend our analysis to the more general case of basic closed semialgebraic sets, not necessarily compact (see {\it e.g.} \cite{blekherman2012truncated} and \cite[]{blekhermanCoreVariety}).

\section{Acknowledgments}
The authors thank Lorenzo Baldi for discussions concerning the degree bounds of
SOS-certificates.
This work is supported by the
{European Commission Marie Sklodowska-Curie Innovative Training Network}
POEMA (Polynomial Optimization, Efficiency through Moments and Algebra, 2019-2023);
by the {Agence Nationale de la
  Recher\-che (ANR)}, grant agreements
{ANR-18-CE33-0011} (SESAME), {ANR-19-CE40-0018}
(De Rerum Natura), {ANR-21-CE48-0006-01} (HYPERSPACE); by the joint
ANR-{Austrian Science Fund FWF} grant agreement {ANR-19-CE48-0015} (ECARP); by the
EOARD-AFOSR grant agreement {FA8665-20-1-7029}.

\end{document}